\documentclass[11pt,a4paper]{article}
\usepackage[margin=1in]{geometry}
\usepackage[T1]{fontenc}
\usepackage[utf8]{inputenc}
\usepackage{lmodern}
\usepackage{microtype}
\usepackage{amsmath,amssymb,amsthm,mathtools}
\usepackage{enumitem}
\usepackage{hyperref}
\usepackage{cleveref}
\usepackage{parskip}

\hypersetup{hidelinks}

\newtheorem{theorem}{Theorem}[section]
\newtheorem{proposition}[theorem]{Proposition}
\newtheorem{lemma}[theorem]{Lemma}
\newtheorem{corollary}[theorem]{Corollary}
\theoremstyle{definition}
\newtheorem{definition}[theorem]{Definition}
\newtheorem{remark}[theorem]{Remark}
\newtheorem{problem}[theorem]{Problem}

\newcommand{\R}{\mathbb{R}}
\newcommand{\dd}{\,\mathrm{d}}
\newcommand{\abs}[1]{\left\lvert #1 \right\rvert}
\DeclareMathOperator{\dist}{dist}
\DeclareMathOperator{\diam}{diam}

\title{Rectifiability of the straight visible\\ boundary and the limits of segment-based Hardy criteria}
\author{H. Abbas \thanks{Faculty of Economics, University of Saida, Algeria \text{a.hafida@yahoo.fr}} \and  A. Azzouz \thanks{Faculty of Science, University of Naama, Algeria. Corresponding author : \text{abdelhalim.azzouz.cus@gmail.com}.}}
\date{}

\begin{document}

\maketitle

\begin{abstract}
Weighted Hardy inequalities on a domain $\Omega\subset\R^n$ can be deduced from lower Hausdorff-content bounds on suitable visible parts of the boundary, and it is natural to look for visible sets described by straight segments, which are easier to verify than the curve-based visual boundary of Koskela and Lehrb\"ack. We show that this approach is subject to an intrinsic obstruction. If a boundary point $\xi$ is seen from $x\in\Omega$ along a segment satisfying a uniform thickness condition, then $\Omega$ contains a truncated cone with vertex $\xi$; consequently the quantitatively straight visible boundary $\partial^{\mathrm{str}}_{x,\alpha}\Omega\cap B(x,C_0d_\Omega(x))$ is a finite union of Lipschitz graphs, is $(n-1)$-rectifiable, and carries no $\lambda$-Hausdorff content for any $\lambda>n-1$. Since content conditions with $\lambda\le n-1$ already imply the $(p,\beta)$-Hardy inequality for $\beta<p-n+\lambda$ without any visibility assumption, segment-based criteria are either unsatisfiable or redundant: they never reach the regime $\beta\ge p-1$ in which visibility is genuinely needed. We show that the exponent $n-1$ is attained, so that the rectifiability theorem is sharp; we exhibit a boundary point lying in the visual boundary of Koskela and Lehrb\"ack but in no $\partial^{\mathrm{str}}_{x,\alpha}\Omega$; and we observe that raw straight visibility, without the thickness condition, escapes the obstruction, is sufficient for $\lambda\le n-1$, and remains open for $\lambda>n-1$.
\end{abstract}

\noindent\textbf{Keywords:} weighted Hardy inequality; visual boundary; Hausdorff content; rectifiability; cone condition; quantitative visibility.

\noindent\textbf{Mathematics Subject Classification (2020):} Primary 26D10; Secondary 46E35, 28A75.

\section{Introduction}\label{sec:intro}

For $1<p<\infty$ and $\beta\in\R$, a domain $\Omega\subset\R^n$ is said to admit the \emph{$(p,\beta)$-Hardy inequality} if there is a constant $C>0$ with
\begin{equation}\label{eq:weighted-hardy}
\int_\Omega \abs{u(x)}^p\, d_\Omega(x)^{\beta-p}\,\dd x
\le C
\int_\Omega \abs{\nabla u(x)}^p\, d_\Omega(x)^\beta\,\dd x
\qquad\text{for every }u\in C_c^\infty(\Omega),
\end{equation}
where $d_\Omega(x):=\dist(x,\partial\Omega)$. The case $p=2$, $\beta=0$ is the classical Hardy inequality.

Modern sufficient conditions for \eqref{eq:weighted-hardy} are geometric and local, and they are read at the scale $d_\Omega(x)$. Two of them frame the present paper. The first requires nothing but a lower Hausdorff-content bound on the boundary: by a theorem of Lehrb\"ack \cite{Lehrback2014}, if
\begin{equation}\label{eq:density-intro}
\mathcal H^\lambda_\infty\bigl(\partial\Omega\cap B(x,2d_\Omega(x))\bigr)\ge C_0\, d_\Omega(x)^\lambda
\qquad\text{for every }x\in\Omega,
\end{equation}
with $0\le\lambda\le n-1$, then $\Omega$ admits the $(p,\beta)$-Hardy inequality for all $\beta<p-n+\lambda$. The second replaces $\partial\Omega$ by the \emph{visual boundary} $v_x(c)\text{--}\partial\Omega$ of Koskela and Lehrb\"ack \cite{KoskelaLehrback2009}, consisting of those boundary points that can be reached from $x$ by a curve obeying a quantitative John-type condition. The two conditions are not interchangeable. Since $\lambda\le n-1$ in \eqref{eq:density-intro} forces $\beta<p-1$, and since the accessibility requirement may be dispensed with altogether when $\beta\le0$ \cite[Section~4]{Lehrback2014}, visibility carries genuine weight exactly in the range
\[
\beta\ge p-1,\qquad\text{equivalently}\qquad \lambda>n-1 .
\]

Because curves are awkward to produce, one would like to certify the hypothesis of \cite{KoskelaLehrback2009} using the simplest possible curve, the segment itself. For $x\in\Omega$ and $\alpha\in(0,1]$, let
\[
\partial^{\mathrm{str}}_{x,\alpha}\Omega
:=\Bigl\{\xi\in\partial\Omega:\ [x,\xi)\subset\Omega
\ \text{ and }\ d_\Omega((1-t)x+t\xi)\ge\alpha\min\{t,1-t\}\abs{x-\xi}\ \ \forall t\in(0,1)\Bigr\},
\]
the \emph{quantitatively straight visible boundary}, and let
\[
\partial_x^{\mathrm{dir}}\Omega:=\{\xi\in\partial\Omega:\ [x,\xi)\subset\Omega\}
\]
be its thickness-free counterpart. The set $\partial^{\mathrm{str}}_{x,\alpha}\Omega$ is defined directly from $\Omega$, requires no auxiliary family of boundary pieces, and is contained in $v_x(\alpha^{-1})\text{--}\partial\Omega$, so that a content bound on it implies the hypothesis of \cite{KoskelaLehrback2009}. The purpose of this paper is to show that this apparent gain is illusory, and to say exactly why.

The mechanism is a cone. Reading the thickness condition near the endpoint $\xi$ shows that $\Omega$ must contain a truncated cone with vertex $\xi$, axis pointing towards $x$, and aperture controlled by $\alpha$ (\Cref{lem:cone}). A set of boundary points each carrying an interior cone of fixed aperture, with axes confined to a small angular sector, is a Lipschitz graph. Splitting the directions from $x$ into finitely many sectors therefore yields our main result.

\begin{theorem}[see \Cref{thm:rectifiable}]
Let $\Omega\subset\R^n$ be a domain, $x\in\Omega$, $\alpha\in(0,1]$ and $C_0\ge1$. Then
$\partial^{\mathrm{str}}_{x,\alpha}\Omega\cap B(x,C_0d_\Omega(x))$
is contained in a union of $N=N(\alpha,C_0,n)$ Lipschitz graphs with constants depending only on $\alpha$. In particular it is $(n-1)$-rectifiable,
\[
\mathcal H^{n-1}\bigl(\partial^{\mathrm{str}}_{x,\alpha}\Omega\cap B(x,C_0d_\Omega(x))\bigr)\lesssim_{\alpha,C_0,n} d_\Omega(x)^{n-1},
\]
and
\[
\mathcal H^\lambda_\infty\bigl(\partial^{\mathrm{str}}_{x,\alpha}\Omega\cap B(x,C_0d_\Omega(x))\bigr)=0
\qquad\text{for every }\lambda>n-1 .
\]
\end{theorem}

The consequence for Hardy inequalities is a dichotomy, stated as \Cref{thm:dichotomy}: a content condition imposed on $\partial^{\mathrm{str}}_{x,\alpha}\Omega$ with exponent $\lambda>n-1$ is satisfied by no domain at all, while for $\lambda\le n-1$ it implies \eqref{eq:density-intro} and therefore yields nothing that the density criterion of \cite{Lehrback2014} does not already give without any visibility assumption. Segment-based criteria are thus unsatisfiable in the regime where visibility matters, and redundant in the regime where they can be satisfied. The exponent $n-1$ is attained --- convex domains and Lipschitz graph domains realise it (\Cref{prop:convex-straight}, \Cref{prop:lipschitz-straight}) --- so the rectifiability theorem is sharp, and the dichotomy is not an artefact of a lossy estimate.

Two further results complete the picture. First, the inclusion $\partial^{\mathrm{str}}_{x,\alpha}\Omega\subset v_x(\alpha^{-1})\text{--}\partial\Omega$ is strict, and strictly so at the level of a single boundary point: in \Cref{prop:strict-inclusion} we construct a planar domain, a point $x$, and a boundary point $\xi$ which is directly visible from $x$ along a segment and is reached by a curve satisfying the visibility condition of \cite{KoskelaLehrback2009} with a finite constant, yet belongs to $\partial^{\mathrm{str}}_{x,\alpha}\Omega$ for no $\alpha\in(0,1]$. Second, the thickness-free set $\partial_x^{\mathrm{dir}}\Omega$ escapes the obstruction entirely, since it carries no cone: a content condition on $\partial_x^{\mathrm{dir}}\Omega$ with $\lambda>n-1$ is not excluded by \Cref{thm:rectifiable}, and whether it implies the $(p,\beta)$-Hardy inequality in that range is, to our knowledge, open (\Cref{prob:raw-fractal}). For $\lambda\le n-1$ it does, for the same reason as above (\Cref{prop:raw-subcritical}).

The paper is organised as follows. \Cref{sec:notation} fixes notation and records the two criteria we use. \Cref{sec:rectifiable} proves the cone property and the rectifiability theorem. \Cref{sec:scope} derives the dichotomy and shows that the exponent $n-1$ is attained. \Cref{sec:strict} contains the strictness example. \Cref{sec:raw} discusses raw straight visibility and states the remaining problem.

\section{Notation and known criteria}\label{sec:notation}

Throughout, $\Omega\subsetneq\R^n$ is a domain, $d_\Omega(x)=\dist(x,\partial\Omega)$, and $B(x,r)$ is the open ball. For $\lambda\ge0$ and $E\subset\R^n$ the $\lambda$-Hausdorff content is
\[
\mathcal H^\lambda_\infty(E):=\inf\Bigl\{\sum_{i=1}^\infty r_i^\lambda\ :\ E\subset\bigcup_{i=1}^\infty B(x_i,r_i)\Bigr\},
\]
and we use repeatedly that $\mathcal H^\lambda_\infty(E)=0$ if and only if $\mathcal H^\lambda(E)=0$. For $x\ne\xi$ we write
\[
[x,\xi):=\{(1-t)x+t\xi:\ 0\le t<1\}.
\]

\begin{definition}[Visible boundaries]\label{def:visible}
Let $x\in\Omega$ and $\alpha\in(0,1]$. The \emph{directly visible boundary} and the \emph{quantitatively straight visible boundary} of $x$ are
\[
\partial_x^{\mathrm{dir}}\Omega:=\{\xi\in\partial\Omega:\ [x,\xi)\subset\Omega\},
\]
\[
\partial^{\mathrm{str}}_{x,\alpha}\Omega
:=\Bigl\{\xi\in\partial_x^{\mathrm{dir}}\Omega:\
d_\Omega\bigl((1-t)x+t\xi\bigr)\ge\alpha\min\{t,1-t\}\abs{x-\xi}\ \text{ for every }t\in(0,1)\Bigr\}.
\]
\end{definition}

We denote by $v_x(c)\text{--}\partial\Omega$ the visual boundary of $x$ with constant $c\ge1$, in the sense of \cite[Section~4]{KoskelaLehrback2009}: the set of $\xi\in\partial\Omega$ that can be joined to $x$ by a rectifiable curve $\gamma:[0,\ell]\to\Omega\cup\{\xi\}$, parametrised by arc length, with $\gamma(0)=x$, $\gamma(\ell)=\xi$, $\ell\le c\abs{x-\xi}$ and
\[
d_\Omega(\gamma(t))\ge c^{-1}\min\{t,\ell-t\}\qquad\text{for }0<t<\ell .
\]
Taking the segment itself as the competing curve gives at once the elementary inclusion
\begin{equation}\label{eq:str-in-visual}
\partial^{\mathrm{str}}_{x,\alpha}\Omega\subset v_x(\alpha^{-1})\text{--}\partial\Omega ,
\end{equation}
which we shall see in \Cref{sec:strict} is strict.

We shall use the following two results as black boxes.

\begin{theorem}[Lehrb\"ack {\cite[Theorem~1.2]{Lehrback2014}}]\label{thm:density-criterion}
Let $\Omega\subset\R^n$ be open and $1<p<\infty$. Suppose there are $0\le\lambda\le n-1$ and $C_0>0$ with
\[
\mathcal H^\lambda_\infty\bigl(\partial\Omega\cap B(x,2d_\Omega(x))\bigr)\ge C_0\,d_\Omega(x)^\lambda
\qquad\text{for every }x\in\Omega .
\]
Then $\Omega$ admits the $(p,\beta)$-Hardy inequality for all $\beta<p-n+\lambda$.
\end{theorem}

\begin{theorem}[Koskela--Lehrb\"ack {\cite[Theorem~1.4]{KoskelaLehrback2009}}]\label{thm:visual-criterion}
Let $1<p<\infty$, let $\Omega\subset\R^n$ be a domain, and assume that there exist $0\le\lambda\le n$, $c\ge1$ and $C>0$ such that
\[
\mathcal H^\lambda_\infty\bigl(v_x(c)\text{--}\partial\Omega\bigr)\ge C\,d_\Omega(x)^\lambda
\qquad\text{for every }x\in\Omega .
\]
Then $\Omega$ admits the $(p,\beta)$-Hardy inequality for all $\beta<p-n+\lambda$.
\end{theorem}

\begin{remark}[Where visibility is needed]\label{rem:where-visibility}
The ranges of \Cref{thm:density-criterion} and \Cref{thm:visual-criterion} differ only through the admissible exponents: the restriction $\lambda\le n-1$ in the former forces $\beta<p-1$. This is not an accident of the proof. Accessibility may be removed altogether when $\beta\le0$, while for $\beta\ge p-1$ the density condition alone is known to be insufficient and accessibility is required; moreover, in $\R^n$ the chain sets used in \cite[Section~4]{Lehrback2014} coincide with suitable visual boundaries in the sense of \cite{KoskelaLehrback2009}. For $\beta=0$ one may also argue through uniform fatness of the complement \cite{Lewis1988,Wannebo1990}. In short, a visibility hypothesis can improve on \Cref{thm:density-criterion} only when it is satisfied with some $\lambda>n-1$.
\end{remark}

\section{The cone property and rectifiability}\label{sec:rectifiable}

Given $\xi\in\R^n$, a unit vector $v$, $\kappa\in(0,1]$ and $h>0$, we write
\[
K(\xi,v,\kappa,h):=\bigl\{\xi+ru:\ 0<r<h,\ u\in S^{n-1},\ \abs{u-v}<\kappa\bigr\}
\]
for the open truncated cone with vertex $\xi$, axis $v$, aperture parameter $\kappa$ and height $h$.

\begin{lemma}[Cone property]\label{lem:cone}
Let $\Omega\subset\R^n$ be a domain, $x\in\Omega$, $\alpha\in(0,1]$ and $\xi\in\partial^{\mathrm{str}}_{x,\alpha}\Omega$. Put $\rho:=\abs{x-\xi}$ and $v:=(x-\xi)/\rho$. Then
\[
K\bigl(\xi,v,\tfrac{\alpha}{2},\tfrac{\rho}{2}\bigr)\subset\Omega .
\]
\end{lemma}

\begin{proof}
For $s\in(0,\tfrac12]$ set $t:=1-s$, so that
\[
(1-t)x+t\xi=\xi+s\rho v .
\]
\Cref{def:visible} gives $\min\{t,1-t\}=s$ and hence
\begin{equation}\label{eq:ball-in-omega}
B(\xi+s\rho v,\ \alpha s\rho)\subset\Omega
\qquad\text{for every }s\in(0,\tfrac12] .
\end{equation}
Let now $y=\xi+ru$ with $0<r<\rho/2$ and $\abs{u-v}<\alpha/2$, and set $s:=r/\rho\in(0,\tfrac12)$. Then
\[
\abs{y-(\xi+s\rho v)}=\abs{ru-rv}=r\abs{u-v}<\tfrac{\alpha}{2}\,r=\tfrac{\alpha}{2}\,s\rho<\alpha s\rho,
\]
so $y\in B(\xi+s\rho v,\alpha s\rho)\subset\Omega$ by \eqref{eq:ball-in-omega}.
\end{proof}

\begin{lemma}[Cone condition implies a Lipschitz graph]\label{lem:cone-to-graph}
Let $v\in S^{n-1}$, $\kappa\in(0,1]$, $h>0$, and let $S\subset\R^n$ be a set with $\diam S<h$ such that
\[
K(\xi,v,\kappa,h)\cap S=\emptyset\qquad\text{for every }\xi\in S .
\]
Then $S$ is contained in the graph of an $L$-Lipschitz function defined on a subset of the hyperplane $v^\perp$, in the coordinates determined by $v$, with
\[
L=L(\kappa):=\frac{1-\kappa^2/2}{\sqrt{1-(1-\kappa^2/2)^2}} .
\]
In particular $\mathcal H^{n-1}(S)\le (1+L^2)^{(n-1)/2}\,\omega_{n-1}\,(\diam S)^{n-1}$, where $\omega_{n-1}$ is the volume of the unit ball of $\R^{n-1}$, and $\mathcal H^\lambda(S)=0$ for every $\lambda>n-1$.
\end{lemma}

\begin{proof}
Let $\xi\ne\xi'$ be points of $S$ and put $u:=(\xi'-\xi)/\abs{\xi'-\xi}$. Since $\abs{\xi'-\xi}\le\diam S<h$, the point $\xi'$ would lie in $K(\xi,v,\kappa,h)$ if $\abs{u-v}<\kappa$; hence $\abs{u-v}\ge\kappa$. Exchanging the roles of $\xi$ and $\xi'$ and noting that $(\xi-\xi')/\abs{\xi-\xi'}=-u$, we get $\abs{u+v}\ge\kappa$ as well. Since $\abs{u\pm v}^2=2\pm2\,u\cdot v$, these two inequalities read
\[
\abs{u\cdot v}\le 1-\tfrac{\kappa^2}{2} .
\]
Writing $\xi'-\xi=a v+w$ with $w\perp v$, so that $a=\abs{\xi'-\xi}\,(u\cdot v)$ and $\abs{w}=\abs{\xi'-\xi}\sqrt{1-(u\cdot v)^2}$, we obtain
\[
\abs{a}\le L(\kappa)\,\abs{w} .
\]
Thus the orthogonal projection $P:\R^n\to v^\perp$, $P\xi:=\xi-(\xi\cdot v)v$, is injective on $S$, and the function $P\xi\mapsto\xi\cdot v$ is $L(\kappa)$-Lipschitz on $P(S)$; that is, $S$ is the graph of an $L(\kappa)$-Lipschitz function over $P(S)\subset v^\perp$. The measure bound follows since an $L$-Lipschitz graph over a set of diameter at most $\diam S$ has $\mathcal H^{n-1}$ measure at most $(1+L^2)^{(n-1)/2}$ times the $\mathcal H^{n-1}$ measure of its base, and $\mathcal H^\lambda$ vanishes on $\mathcal H^{n-1}$-finite sets when $\lambda>n-1$.
\end{proof}

\begin{theorem}[Rectifiability of the straight visible boundary]\label{thm:rectifiable}
Let $\Omega\subset\R^n$ be a domain, $x\in\Omega$, $\alpha\in(0,1]$ and $C_0\ge1$. Then the set
\[
S_x:=\partial^{\mathrm{str}}_{x,\alpha}\Omega\cap B\bigl(x,C_0d_\Omega(x)\bigr)
\]
is contained in the union of at most $N=N(\alpha,C_0,n)$ Lipschitz graphs, each with Lipschitz constant at most $L(\alpha/4)$ as in \Cref{lem:cone-to-graph}. Consequently $S_x$ is $(n-1)$-rectifiable,
\[
\mathcal H^{n-1}(S_x)\le C(\alpha,C_0,n)\,d_\Omega(x)^{n-1},
\]
and
\[
\mathcal H^\lambda_\infty(S_x)=0\qquad\text{for every }\lambda>n-1 .
\]
\end{theorem}

\begin{proof}
Write $d:=d_\Omega(x)$ and, for $\xi\in S_x$, let $v_\xi:=(x-\xi)/\abs{x-\xi}$ and $\rho_\xi:=\abs{x-\xi}$. Since $\xi\in\partial\Omega$ we have $\rho_\xi\ge d$, and since $\xi\in B(x,C_0d)$ we have $\rho_\xi<C_0 d$. By \Cref{lem:cone},
\begin{equation}\label{eq:cone-uniform}
K\bigl(\xi,v_\xi,\tfrac\alpha2,\tfrac d2\bigr)\subset K\bigl(\xi,v_\xi,\tfrac\alpha2,\tfrac{\rho_\xi}2\bigr)\subset\Omega
\qquad\text{for every }\xi\in S_x ,
\end{equation}
so all the cones may be truncated at the common height $h:=d/2$.

Cover $S^{n-1}$ by finitely many sets $\Sigma_1,\dots,\Sigma_{N_1}$ of diameter less than $\alpha/4$, with $N_1=N_1(\alpha,n)$, and choose $v_j\in\Sigma_j$. Cover $B(x,C_0d)$ by finitely many balls $D_1,\dots,D_{N_2}$ of radius $d/8$, with $N_2=N_2(C_0,n)$. For $1\le j\le N_1$ and $1\le i\le N_2$ put
\[
S_{j,i}:=\{\xi\in S_x\cap D_i:\ v_\xi\in\Sigma_j\},
\]
so that $S_x=\bigcup_{j,i}S_{j,i}$ and $N:=N_1N_2$ depends only on $\alpha$, $C_0$ and $n$.

Fix $j,i$ and let $\xi\in S_{j,i}$. If $u\in S^{n-1}$ satisfies $\abs{u-v_j}<\alpha/4$, then
\[
\abs{u-v_\xi}\le\abs{u-v_j}+\abs{v_j-v_\xi}<\tfrac\alpha4+\tfrac\alpha4=\tfrac\alpha2 ,
\]
because $v_\xi,v_j\in\Sigma_j$ and $\diam\Sigma_j<\alpha/4$. Hence
\[
K\bigl(\xi,v_j,\tfrac\alpha4,h\bigr)\subset K\bigl(\xi,v_\xi,\tfrac\alpha2,h\bigr)\subset\Omega
\]
by \eqref{eq:cone-uniform}. Since $\Omega$ is open and $S_{j,i}\subset\partial\Omega$, these cones contain no point of $S_{j,i}$:
\[
K\bigl(\xi,v_j,\tfrac\alpha4,h\bigr)\cap S_{j,i}=\emptyset\qquad\text{for every }\xi\in S_{j,i}.
\]
Moreover $\diam S_{j,i}\le\diam D_i=d/4<h$. \Cref{lem:cone-to-graph}, applied with $v=v_j$ and $\kappa=\alpha/4$, shows that $S_{j,i}$ is contained in a Lipschitz graph with constant $L(\alpha/4)$, and that
\[
\mathcal H^{n-1}(S_{j,i})\le\bigl(1+L(\alpha/4)^2\bigr)^{(n-1)/2}\omega_{n-1}\,(d/4)^{n-1}.
\]
Summing over the $N$ pairs $(j,i)$ gives the stated bound for $\mathcal H^{n-1}(S_x)$, and $(n-1)$-rectifiability. Finally, a finite union of Lipschitz graphs has $\sigma$-finite $\mathcal H^{n-1}$ measure, so $\mathcal H^\lambda(S_x)=0$ for $\lambda>n-1$, and therefore $\mathcal H^\lambda_\infty(S_x)=0$.
\end{proof}

\begin{remark}\label{rem:no-cone-for-dir}
Only the thickness condition is used, and it is used only near the endpoint: what drives \Cref{thm:rectifiable} is the behaviour of $t\mapsto d_\Omega((1-t)x+t\xi)$ as $t\uparrow1$. The set $\partial_x^{\mathrm{dir}}\Omega$ carries no cone, and \Cref{thm:rectifiable} says nothing about it; see \Cref{sec:raw}.
\end{remark}

\section{The scope of segment-based criteria}\label{sec:scope}

We first record that the radius $C_0$ in a density condition is immaterial.

\begin{lemma}[Radius normalisation]\label{lem:radius}
Let $\Omega\subsetneq\R^n$ be a domain, $\lambda\ge0$, $c_0>0$ and $C_0\ge2$, and suppose that
\[
\mathcal H^\lambda_\infty\bigl(\partial\Omega\cap B(x,C_0d_\Omega(x))\bigr)\ge c_0\,d_\Omega(x)^\lambda
\qquad\text{for every }x\in\Omega .
\]
Then
\[
\mathcal H^\lambda_\infty\bigl(\partial\Omega\cap B(x,2d_\Omega(x))\bigr)\ge c_0C_0^{-\lambda}\,d_\Omega(x)^\lambda
\qquad\text{for every }x\in\Omega .
\]
\end{lemma}

\begin{proof}
Fix $x\in\Omega$ and write $d:=d_\Omega(x)$. Since $\partial\Omega$ is closed and nonempty, there is $\xi^*\in\partial\Omega$ with $\abs{x-\xi^*}=d$. Let
\[
y:=x+\Bigl(1-\tfrac1{C_0}\Bigr)(\xi^*-x),
\]
so that $\abs{y-\xi^*}=d/C_0$ and $\abs{x-y}=d(1-1/C_0)$. Then $d_\Omega(y)\le\abs{y-\xi^*}=d/C_0$, while $d_\Omega(y)\ge d-\abs{x-y}=d/C_0$; hence $d_\Omega(y)=d/C_0$. Applying the hypothesis at $y$,
\[
\mathcal H^\lambda_\infty\bigl(\partial\Omega\cap B(y,C_0d_\Omega(y))\bigr)\ge c_0 d_\Omega(y)^\lambda=c_0C_0^{-\lambda}d^\lambda .
\]
Finally $C_0d_\Omega(y)=d$ and $\abs{x-y}<d$, so $B(y,C_0d_\Omega(y))\subset B(x,2d)$, and the conclusion follows by monotonicity of the content.
\end{proof}

We can now state the dichotomy. By a \emph{segment-based criterion} we mean the hypothesis
\begin{equation}\label{eq:str-hypothesis}
\mathcal H^\lambda_\infty\bigl(\partial^{\mathrm{str}}_{x,\alpha}\Omega\cap B(x,C_0d_\Omega(x))\bigr)\ge c_0\,d_\Omega(x)^\lambda
\qquad\text{for every }x\in\Omega ,
\end{equation}
with constants $c_0>0$, $C_0\ge2$, $\alpha\in(0,1]$ and an exponent $0\le\lambda\le n$.

\begin{theorem}[Unsatisfiable or redundant]\label{thm:dichotomy}
Let $\Omega\subsetneq\R^n$ be a domain and let $1<p<\infty$.
\begin{enumerate}[label=(\roman*)]
\item If $\lambda>n-1$, then no domain satisfies \eqref{eq:str-hypothesis}, for any choice of $c_0>0$, $C_0\ge2$ and $\alpha\in(0,1]$.
\item If $0\le\lambda\le n-1$ and $\Omega$ satisfies \eqref{eq:str-hypothesis}, then $\Omega$ satisfies the density condition of \Cref{thm:density-criterion} with the same exponent $\lambda$ and constant $c_0C_0^{-\lambda}$, and consequently admits the $(p,\beta)$-Hardy inequality for all $\beta<p-n+\lambda$.
\end{enumerate}
In particular, a segment-based criterion never yields a $(p,\beta)$-Hardy inequality with $\beta\ge p-1$, and in the range $\beta<p-1$ where it applies it yields nothing beyond \Cref{thm:density-criterion}, whose hypothesis involves no visibility whatsoever.
\end{theorem}

\begin{proof}
(i) Let $x\in\Omega$; then $d_\Omega(x)>0$, while $\mathcal H^\lambda_\infty(\partial^{\mathrm{str}}_{x,\alpha}\Omega\cap B(x,C_0d_\Omega(x)))=0$ by \Cref{thm:rectifiable}. Hence \eqref{eq:str-hypothesis} fails at every point of $\Omega$.

(ii) Since $\partial^{\mathrm{str}}_{x,\alpha}\Omega\subset\partial\Omega$, the hypothesis \eqref{eq:str-hypothesis} implies
\[
\mathcal H^\lambda_\infty\bigl(\partial\Omega\cap B(x,C_0d_\Omega(x))\bigr)\ge c_0 d_\Omega(x)^\lambda
\qquad\text{for every }x\in\Omega,
\]
and \Cref{lem:radius} converts this into the hypothesis of \Cref{thm:density-criterion} with constant $c_0C_0^{-\lambda}$. The last assertion follows since $\lambda\le n-1$ forces $p-n+\lambda\le p-1$.
\end{proof}

\begin{remark}\label{rem:reading}
\Cref{thm:dichotomy} should be read together with \Cref{rem:where-visibility}. Visibility hypotheses are useful precisely when they hold with $\lambda>n-1$, and that is exactly what a segment, subject to a uniform thickness condition, can never deliver: by \Cref{lem:cone} the thickness condition manufactures an interior cone at every visible boundary point, and cones force $(n-1)$-rectifiability. The obstruction is therefore geometric and not an artefact of the passage through \Cref{thm:visual-criterion}: it is intrinsic to the set $\partial^{\mathrm{str}}_{x,\alpha}\Omega$ itself.
\end{remark}

\subsection{The exponent \texorpdfstring{$n-1$}{n-1} is attained}\label{sec:sharp}

\Cref{thm:rectifiable} would be vacuous if $\partial^{\mathrm{str}}_{x,\alpha}\Omega$ were always small. It is not: the two standard model classes realise $\lambda=n-1$, so that the bound of \Cref{thm:rectifiable} is attained and the dichotomy of \Cref{thm:dichotomy} is sharp.

\begin{proposition}[Convex domains]\label{prop:convex-straight}
Let $\Omega\subset\R^n$ be a bounded convex domain and let $C_0\ge1$. Then, with $\alpha:=1/C_0$,
\[
\partial\Omega\cap B(x,C_0 d_\Omega(x))\subset\partial^{\mathrm{str}}_{x,\alpha}\Omega
\qquad\text{for every }x\in\Omega .
\]
\end{proposition}

\begin{proof}
Fix $x\in\Omega$ and set $r:=d_\Omega(x)$, so that $B(x,r)\subset\Omega$ by convexity. Let $\xi\in\partial\Omega\cap B(x,C_0r)$ and put $z_t:=(1-t)x+t\xi$ for $t\in(0,1)$; by convexity $[x,\xi)\subset\Omega$. Every point of $B(z_t,(1-t)r)$ can be written as $(1-t)y+t\xi$ with $y\in B(x,r)$, and convexity gives $(1-t)y+t\xi\in\Omega$ for $t<1$; hence
\[
d_\Omega(z_t)\ge(1-t)r\ge\frac{1}{C_0}(1-t)\abs{x-\xi},
\]
using $\abs{x-\xi}\le C_0r$. If $0<t\le\frac12$ then $1-t\ge t$, so $d_\Omega(z_t)\ge\frac1{C_0}t\abs{x-\xi}$; if $\frac12\le t<1$ the displayed bound is the required one. In both cases
\[
d_\Omega(z_t)\ge\frac1{C_0}\min\{t,1-t\}\abs{x-\xi},
\]
that is, $\xi\in\partial^{\mathrm{str}}_{x,1/C_0}\Omega$.
\end{proof}

\begin{corollary}[Balls]\label{cor:balls}
If $\Omega=B(x_0,R)$ then for every $x\in\Omega$ and $C_0\ge1$,
\[
\mathcal H^{n-1}_\infty\bigl(\partial^{\mathrm{str}}_{x,1/C_0}\Omega\cap B(x,C_0d_\Omega(x))\bigr)\asymp d_\Omega(x)^{n-1},
\]
so that \eqref{eq:str-hypothesis} holds with $\lambda=n-1$, and the upper bound of \Cref{thm:rectifiable} is attained.
\end{corollary}

\begin{proof}
Balls are convex, so \Cref{prop:convex-straight} gives $\partial\Omega\cap B(x,C_0d_\Omega(x))\subset\partial^{\mathrm{str}}_{x,1/C_0}\Omega$. The set $\partial\Omega\cap B(x,C_0d_\Omega(x))$ is a spherical cap of radius comparable to $d_\Omega(x)$, whence the content is comparable to $d_\Omega(x)^{n-1}$; the matching upper bound is \Cref{thm:rectifiable}.
\end{proof}

\begin{proposition}[Lipschitz graph domains]\label{prop:lipschitz-straight}
Let $\varphi:\R^{n-1}\to\R$ be $L$-Lipschitz and let $\Omega_\varphi:=\{(y',y_n):y_n>\varphi(y')\}$. Fix $x=(x',x_n)\in\Omega_\varphi$ and set $r:=x_n-\varphi(x')>0$. Then
\[
\frac{r}{1+L}\le d_{\Omega_\varphi}(x)\le r,
\]
and, with $\eta:=\alpha:=\frac1{4(1+L)}$ and
\[
E_x:=\bigl\{(\xi',\varphi(\xi'))\in\partial\Omega_\varphi:\ \abs{\xi'-x'}\le\eta r\bigr\},
\]
one has $E_x\subset\partial^{\mathrm{str}}_{x,\alpha}\Omega_\varphi$ and $\mathcal H^{n-1}_\infty(E_x)\gtrsim_{n,L}d_{\Omega_\varphi}(x)^{n-1}$. Thus Lipschitz graph domains satisfy \eqref{eq:str-hypothesis} locally with $\lambda=n-1$.
\end{proposition}

\begin{proof}
The bound $d_{\Omega_\varphi}(x)\le r$ is immediate from $(x',\varphi(x'))\in\partial\Omega_\varphi$. For the lower bound put $\rho:=r/(1+L)$; if $z=(z',z_n)\in B(x,\rho)$ then $z_n>x_n-\rho$ and $\varphi(z')\le\varphi(x')+L\rho$, so
\[
z_n-\varphi(z')\ge x_n-\varphi(x')-(1+L)\rho=0 ,
\]
and in fact $z_n>\varphi(z')$ since the ball is open; hence $B(x,\rho)\subset\Omega_\varphi$ and $d_{\Omega_\varphi}(x)\ge\rho$.

Let $\xi=(\xi',\varphi(\xi'))\in E_x$ and $z_t:=(1-t)x+t\xi$ for $0<t<1$. By the Lipschitz property,
\[
\varphi(z_t')\le\varphi(\xi')+L\abs{z_t'-\xi'}=\varphi(\xi')+L(1-t)\abs{x'-\xi'},
\]
whence
\[
z_{t,n}-\varphi(z_t')\ge(1-t)\bigl(x_n-\varphi(\xi')-L\abs{x'-\xi'}\bigr)
\ge(1-t)\bigl(r-2L\abs{x'-\xi'}\bigr)\ge\tfrac12(1-t)r>0,
\]
using $\varphi(\xi')\le\varphi(x')+L\abs{x'-\xi'}$ and $2L\abs{\xi'-x'}\le2L\eta r\le\frac12r$. In particular $[x,\xi)\subset\Omega_\varphi$, and the same argument as above converts the height estimate into
\[
d_{\Omega_\varphi}(z_t)\ge\frac{1}{1+L}\cdot\frac{(1-t)r}{2}=\frac{(1-t)r}{2(1+L)} .
\]
Since $(1+L)\eta=\frac14$ we have $\abs{x-\xi}\le\abs{x'-\xi'}+\abs{x_n-\varphi(\xi')}\le\eta r+r+L\eta r\le2r$, so
\[
d_{\Omega_\varphi}(z_t)\ge\frac{1}{4(1+L)}\min\{t,1-t\}\abs{x-\xi},
\]
that is, $\xi\in\partial^{\mathrm{str}}_{x,\alpha}\Omega_\varphi$. Finally the projection $\pi(y',y_n)=y'$ is $1$-Lipschitz and $\pi(E_x)=\overline{B^{n-1}(x',\eta r)}$, so $\mathcal H^{n-1}_\infty(E_x)\gtrsim_n(\eta r)^{n-1}\gtrsim_{n,L}r^{n-1}$, and $r$ is comparable to $d_{\Omega_\varphi}(x)$.
\end{proof}

\begin{remark}\label{rem:lipschitz-local}
\Cref{prop:lipschitz-straight} is stated for a global epigraph, but the same geometry appears locally in bounded Lipschitz domains after flattening the boundary in a coordinate chart. Together with \Cref{thm:rectifiable} this says that, for these classes, $\partial^{\mathrm{str}}_{x,\alpha}\Omega$ is exactly of dimension $n-1$: no more by rectifiability, no less by the above. The corresponding Hardy range is $\beta<p-1$, which for Lipschitz domains is the classical one.
\end{remark}

\section{Strictness of the inclusion into the visual boundary}\label{sec:strict}

The inclusion \eqref{eq:str-in-visual} cannot be reversed, and the failure is already visible at a single boundary point. The example also illustrates \Cref{rem:no-cone-for-dir}: the point $\xi$ below lies in $\partial_x^{\mathrm{dir}}\Omega$ and has no interior cone at all.

\begin{proposition}[Strictness]\label{prop:strict-inclusion}
There exist a bounded planar domain $\Omega\subset\R^2$, a point $x\in\Omega$ and a boundary point $\xi\in\partial\Omega$ such that
\begin{enumerate}[label=(\roman*)]
\item $\xi\in\partial_x^{\mathrm{dir}}\Omega$;
\item $\xi\notin\partial^{\mathrm{str}}_{x,\alpha}\Omega$ for every $\alpha\in(0,1]$;
\item $\xi\in v_x(4)\text{--}\partial\Omega$.
\end{enumerate}
\end{proposition}

\begin{proof}
Let $R:=(0,1)\times(-1,1)$. For $k\ge1$ set
\[
t_k:=1-\frac{1}{k+1},\qquad \delta_k:=2^{-k},
\]
and define the closed vertical segments (``teeth'') $T_k:=\{t_k\}\times[\delta_k,\tfrac12]$. Since the $t_k$ are pairwise distinct, the $T_k$ are pairwise disjoint closed sets, so
\[
\Omega:=R\setminus\bigcup_{k\ge1}T_k
\]
is open; it is connected because every $T_k$ lies at height $y\ge\delta_k>0$, so any two points of $\Omega$ can be joined by a path through the strip $\{y=-\tfrac12\}$, which meets no $T_k$. Fix $x:=(\tfrac14,0)\in\Omega$ and $\xi:=(1,0)$. Every point of $\Omega$ has first coordinate $<1$, and $(t_k,0)\in\Omega$ for every $k$ by (i) below, with $t_k\to1$; hence $\xi\in\partial\Omega$.

\emph{(i)} The segment $[x,\xi)=\{(t,0):\tfrac14\le t<1\}$ lies on the line $y=0$, while $T_k\subset\R\times[\delta_k,\tfrac12]$ with $\delta_k>0$; hence $[x,\xi)\subset\Omega$ and $\xi\in\partial_x^{\mathrm{dir}}\Omega$.

\emph{(ii)} For $k\ge6$ one has $2^{-k}<[(k+1)(k+2)]^{-1}$ (at $k=6$, $2^{-6}=1/64<1/56$; the left side decays exponentially and the right side only polynomially). For such $k$ the nearest point of $\partial\Omega$ to $(t_k,0)$ is the tip $(t_k,\delta_k)$ of $T_k$: for $j\ne k$ the segment $T_j$ lies at distance at least $\abs{t_k-t_j}\ge\min\bigl(\tfrac1{k(k+1)},\tfrac1{(k+1)(k+2)}\bigr)=\tfrac1{(k+1)(k+2)}>\delta_k$, and the edges of $R$ are at distance at least $\min(t_k,1-t_k,1)=1/(k+1)\gg\delta_k$. Hence $d_\Omega(t_k,0)=\delta_k=2^{-k}$. Writing $(t_k,0)=(1-s_k)x+s_k\xi$ gives $s_k=(4t_k-1)/3$, so $1-s_k=\tfrac{4}{3(k+1)}=\min\{s_k,1-s_k\}$, while $\abs{x-\xi}=\tfrac34$. If $\xi$ belonged to $\partial^{\mathrm{str}}_{x,\alpha}\Omega$, then \Cref{def:visible} at $t=s_k$ would give
\[
2^{-k}=d_\Omega(t_k,0)\ge\alpha\cdot\frac{4}{3(k+1)}\cdot\frac34=\frac{\alpha}{k+1},
\]
that is $\alpha\le(k+1)2^{-k}$ for every $k\ge6$. Since $(k+1)2^{-k}\to0$, no $\alpha\in(0,1]$ qualifies.

\emph{(iii)} Let $\gamma$ be the arc-length parametrised polygonal path from $x=(\tfrac14,0)$ through $(\tfrac14,-\tfrac14)$ and $(\tfrac34,-\tfrac14)$ to $\xi=(1,0)$, of total length
\[
\ell=\frac14+\frac12+\frac{\sqrt2}{4}=\frac{3+\sqrt2}{4}<3=4\abs{x-\xi}.
\]
On the first piece every point has first coordinate $\tfrac14$, hence horizontal distance at least $t_1-\tfrac14=\tfrac14$ from every tooth and from the left edge of $R$, so $d_\Omega\ge\tfrac14$ there. On the second piece every point has second coordinate $-\tfrac14$, hence vertical distance at least $\delta_k+\tfrac14>\tfrac14$ from every $T_k$, and horizontal distance at least $\tfrac14$ from the vertical edges of $R$; again $d_\Omega\ge\tfrac14$. On the third piece, parametrised by $\sigma\in[0,1]$ as $\bigl(\tfrac34+\tfrac\sigma4,-\tfrac14+\tfrac\sigma4\bigr)$, the second coordinate is $\le0<\delta_k$ for every $k$, so the distance to $T_k$ is at least $\delta_k+\tfrac{1-\sigma}{4}>\tfrac{1-\sigma}4$, and the distance to the right edge of $R$ equals $\tfrac{1-\sigma}4$; since the remaining arc length to $\xi$ is $(1-\sigma)\tfrac{\sqrt2}{4}$,
\[
d_\Omega(\gamma(t))\ge\frac{1-\sigma}{4}=\frac1{\sqrt2}(\ell-t)\ge\frac14(\ell-t).
\]
As $\ell/2<\tfrac34$, one has $\min\{t,\ell-t\}=\ell-t$ on the third piece and $\min\{t,\ell-t\}\le\ell/2<1$ on the first two, where $d_\Omega\ge\tfrac14$; hence $d_\Omega(\gamma(t))\ge\tfrac14\min\{t,\ell-t\}$ for all $t\in(0,\ell)$. Together with $\ell\le4\abs{x-\xi}$ this gives $\xi\in v_x(4)\text{--}\partial\Omega$.
\end{proof}

\begin{remark}\label{rem:strict-vs-cone}
By \Cref{lem:cone}, assertion (ii) is equivalent to the absence of an interior cone at $\xi$ with axis pointing towards $x$, and this is what the computation in (ii) verifies quantitatively: the teeth descend to height $\delta_k=2^{-k}$ at horizontal distance $1/(k+1)$ from $\xi$, and $2^{-k}(k+1)\to0$. Assertion (iii) shows that a curve may nevertheless avoid the teeth at bounded cost, since the region below them is free of obstructions. The two assertions together say that between the segment and an arbitrary John-type curve there is a genuine geometric gap.
\end{remark}

\section{Raw straight visibility}\label{sec:raw}

We finally consider $\partial_x^{\mathrm{dir}}\Omega$, where no thickness is imposed. Here \Cref{thm:rectifiable} does not apply, and the situation is genuinely different. We first note that below the critical exponent the criterion is valid, and for the same reason as in \Cref{thm:dichotomy}(ii): visibility plays no role.

\begin{proposition}[Subcritical raw visibility]\label{prop:raw-subcritical}
Let $\Omega\subsetneq\R^n$ be a domain, $1<p<\infty$, and suppose that there are $c_0>0$, $C_0\ge2$ and $0\le\lambda\le n-1$ with
\begin{equation}\label{eq:raw-hypothesis}
\mathcal H^\lambda_\infty\bigl(\partial_x^{\mathrm{dir}}\Omega\cap B(x,C_0d_\Omega(x))\bigr)\ge c_0\,d_\Omega(x)^\lambda
\qquad\text{for every }x\in\Omega .
\end{equation}
Then $\Omega$ admits the $(p,\beta)$-Hardy inequality for every $\beta<p-n+\lambda$.
\end{proposition}

\begin{proof}
Since $\partial_x^{\mathrm{dir}}\Omega\subset\partial\Omega$, the hypothesis \eqref{eq:raw-hypothesis} gives
$\mathcal H^\lambda_\infty(\partial\Omega\cap B(x,C_0d_\Omega(x)))\ge c_0 d_\Omega(x)^\lambda$ for every $x\in\Omega$. By \Cref{lem:radius} the density condition of \Cref{thm:density-criterion} holds with constant $c_0C_0^{-\lambda}$, and the conclusion follows.
\end{proof}

The interesting range is therefore $\lambda>n-1$, and there the obstruction of \Cref{sec:rectifiable} is absent. Indeed, $\partial_x^{\mathrm{dir}}\Omega$ is a radial graph --- if $\xi,\xi'\in\partial_x^{\mathrm{dir}}\Omega$ lie on a common ray from $x$ with $\abs{x-\xi}<\abs{x-\xi'}$, then $\xi\in[x,\xi')\subset\Omega$, which is impossible --- but the radial function is subject to no regularity whatsoever, and radial graphs of irregular functions may have Hausdorff dimension larger than $n-1$. Fractal boundaries that are radially accessible are the natural candidates: the von Koch snowflake, whose boundary is Ahlfors regular of dimension $\log4/\log3>1$, is the first one to examine.

\begin{problem}\label{prob:raw-fractal}
Let $\lambda>n-1$.
\begin{enumerate}[label=(\alph*)]
\item Does there exist a domain $\Omega\subset\R^n$ satisfying \eqref{eq:raw-hypothesis} with this $\lambda$? Equivalently, can the directly visible boundary carry $\lambda$-content at the scale $d_\Omega(x)$ uniformly in $x$?
\item If so, does \eqref{eq:raw-hypothesis} imply the $(p,\beta)$-Hardy inequality for $\beta<p-n+\lambda$, that is, in the range $\beta\ge p-1$ where the density criterion of \Cref{thm:density-criterion} is unavailable?
\end{enumerate}
\end{problem}

A positive answer to both parts would exhibit the first genuinely segment-based criterion reaching $\beta\ge p-1$, and by \Cref{thm:rectifiable} such a criterion would necessarily be thickness-free. A negative answer to (a) would show that the phenomenon isolated in \Cref{thm:rectifiable} survives the removal of the thickness condition, and that straight lines are intrinsically unable to see a fractal amount of boundary; a negative answer to (b) would identify the thickness condition as the exact price of accessibility.

\begin{remark}\label{rem:final}
\Cref{prob:raw-fractal}(b) should be compared with \Cref{prop:strict-inclusion}, which shows that a directly visible point need not be quantitatively visible along the segment, but may still be reached by a curve. Whether raw straight visibility forces John-type accessibility in the quantitative sense required by \Cref{thm:visual-criterion} --- uniformly in $x$, which is what the Hardy inequality needs --- is precisely the geometric content of \Cref{prob:raw-fractal}(b).
\end{remark}

\section{Concluding remarks}\label{sec:concluding}

The results above delimit the reach of straight-line visibility in the theory of weighted Hardy inequalities. A uniform thickness condition along a segment is not a technical convenience: it manufactures an interior cone at the endpoint, and cones force $(n-1)$-rectifiability. Consequently the quantitatively straight visible boundary is always, at the scale $d_\Omega(x)$, an $(n-1)$-dimensional object, and any criterion built on it lives in the range $\beta<p-1$, where boundary density alone already suffices. What is lost is exactly the regime $\beta\ge p-1$, for which the curve-based visual boundary of \cite{KoskelaLehrback2009} was designed and in which it cannot be replaced by segments.

Two directions remain. The first is \Cref{prob:raw-fractal}: without the thickness condition the cone disappears, the rectifiability obstruction with it, and the question of what straight lines can see becomes a question about radial graphs over fractal boundaries. The second concerns the fractional setting, where visibility is already known to play an essential role \cite{DydaLehrbackVahakangas2014,IhnatsyevaLehrbackTuominenVahakangas2014}; since \Cref{lem:cone} and \Cref{thm:rectifiable} are purely geometric statements about $\partial^{\mathrm{str}}_{x,\alpha}\Omega$, they apply verbatim there, and the same dichotomy should be expected for any fractional criterion formulated in terms of quantitatively straight visible sets.

\subsection*{Authors' contributions}
All authors contributed equally to this work. All authors read and approved the final manuscript.
\subsection*{Funding} The authors declare that there is no source of funding for this research.
\subsection*{Availability of data and materials} Data sharing not applicable to this paper as no data sets were generated or analyzed during the current study.
\subsection*{Competing interests}
The authors declare that they have no competing interests.

\end{document}